\documentclass[a4paper,10pt,twoside]{amsart}
\usepackage{amsthm}
\usepackage{amsmath}
\usepackage{amssymb}
\usepackage{amsfonts}
\usepackage{amscd}
\usepackage[english]{babel}
\usepackage{stmaryrd}
\usepackage{enumerate}

\newtheorem{thm}{Theorem}
\newtheorem{cor}{Corollary}
\newtheorem{prop}{Proposition}
\newtheorem*{rem}{Remark}

\newtheorem{exam}{Example}

\newtheorem*{defi}{Definition}
\newtheorem{lem}{Lemma}

\begin{document}

    \title{How to Solve the Matrix Equation $XA-AX=f(X)$}
    \author{Gerald BOURGEOIS}

    \address{G\'erald Bourgeois, GAATI, Universit\'e de la polyn\'esie fran\c caise, BP 6570, 98702 FAA'A, Tahiti, Polyn\'esie Fran\c caise.}
    \email{gerald.bourgeois@upf.pf}
        

\begin{abstract}
Let $f$ be an analytic function defined on a complex domain $\Omega$ and $A\in\mathcal{M}_n(\mathbb{C})$. We assume that there exists a unique $\alpha$ satisfying $f(\alpha)=0$. When $f'(\alpha)=0$ and $A$ is non-derogatory, we completely solve the equation $XA-AX=f(X)$. This generalizes Burde's results. When $f'(\alpha)\not=0$, we give a method to solve completely the equation $XA-AX=f(X)$: we reduce the problem to solving a sequence of Sylvester equations. Solutions of the equation $f(XA-AX)=X$ are also given in particular cases.
\end{abstract}

\maketitle
%//////////////////////////////////////////////////////////////////////////////////////////////////////////////////////////////////////////////////
    \section{Introduction}
    Usually, the exact solutions of algebraic matrix equations are obtained as limits of approximating solutions. For instance, this method is used to find positive solutions of the equation $X+BX^{-1}B^*=A$, where $A$ is a hermitian matrix (see \cite{4}). Riccati equation $XAX+XB+CX+D=0$ or its hermitian counterpart $XAX+XB+B^*X+C=0$ ($A,C$ hermitian matrices) and the quadratic equation $AX^2+BX+C=0$ are also solved in this way (see \cite{5},\cite{6} and \cite{7}). On the other hand, very few results are known about explicit solutions of algebraic matrix equations. In \cite{9},the authors give general complete parametric forms for the solutions $(X,Y)$ of the generalized Sylvester equation $AX-XF=BY$. In \cite{10}, the author finds exact rational solutions of the equation $p(X)=A$ where $p$ is a polynomial. Finally, the author in \cite{8} found exact solutions of the system $\{A+B+C=\alpha{I_n},A^2+B^2+C^2=\beta{I_n}, A^3+B^3+C^3=\gamma{I_n}\},\alpha,\beta,\gamma\in\mathbb{C},n\in\mathbb{N}_{\geq{2}}$, where the $(n\times{n})$ matrices $A,B,C$ are to be determined.\\    
\indent Let $n\in\mathbb{N}_{\geq{2}}$, $K$ be a field and $A\in\mathcal{M}_n(K)$. In \cite{11}, the equation $XA-AX=\tau(X)$, where $\tau$ is a $K$-automorphism of finite order is studied. In \cite{1}, Burde completely solved the related matrix equation $XA-AX=X^p$, with $p\in\mathbb{N}_{\geq{2}}$, unknown $X\in\mathcal{M}_n(\mathbb{C})$ and $A$ a given $(n\times{n})$ non-derogatory complex matrix. In this article, we propose to extend Burde's results to a more general class of matrix equations.
  \indent We introduce notations that will be used in the sequel of the article.\\
    \textbf{Notation}. $i)$ For any $X\in\mathcal{M}_n(\mathbb{C})$, $\sigma(X)$ denotes the spectrum of $X$.\\
  $ii)$ Denote by $I_n$ the identity matrix of $\mathcal{M}_n(\mathbb{C})$.\\\\    
     Let $\Omega$ be a complex domain and $f:\Omega\rightarrow\mathbb{C}$ be an analytic function.   
    We consider the matrix equation
  \begin{equation}   \label{equation 1}   XA-AX=f(X)   \end{equation}
      where the $(n\times{n})$ complex matrix $A$ is given  and the unknown is a $(n\times{n})$ complex matrix $X$ such that $\sigma(X)\subset{\Omega}$. We assume that there exists a unique $\alpha\in{\Omega}$ such that $f(\alpha)=0$.\\
 \indent When $f'(\alpha)=0$ and $A$ is non-derogatory, we completely solve Equation (\ref{equation 1}). The solution of this problem follows Burde's method.\\ 
\indent When $f'(\alpha)\not=0$ and there exist two eigenvalues of $A$ whose difference is $f'(\alpha)$, we prove that Equation (\ref{equation 1}) admits non trivial solutions. Moreover, we give a method to completely solve Equation (\ref{equation 1}). Indeed we reduce the problem to solving a sequence of Sylvester equations. We apply this to the equation $$XA-AX=\log(X).$$
Now it should be noted that the method used to prove these results differs from that of Burde.\\
\indent We have a look at the equation $f(XA-AX)=X$ when $f$ is locally invertible in a neighborhood of $0$. In particular we show that the equations $XA-AX=\log(X)$ and $e^{XA-AX}=X$ have same solutions.
  Further results on the equation
       $$f(XA-AX)=X$$
       are also given in the case where $f$ is not locally invertible in any neighborhood of $0$.
      %//////////////////////////////////////////////////////////////////////////////////////////////////////////////////////////////////////////////////
      
      \section{General remarks}
      Recall the following definitions
     \begin{defi} 
  i) Let $A,B\in\mathcal{M}_n(\mathbb{C})$. The matrices $A,B$ are said \emph{simultaneously triangularizable} if there exists an invertible matrix $P\in\mathcal{M}_n(\mathbb{C})$ such that $P^{-1}AP$ and $P^{-1}BP$ are upper triangular.\\ 
 ii) \cite{2} Let $X\in\mathcal{M}_n(\mathbb{C})$ be such that $\sigma(X)\subset{\Omega}$. We put $$f(X)=\dfrac{1}{2i\pi}\int_{\Gamma}f(z)(zI_n-X)^{-1}dz,$$
    where $\Gamma$ is a counterclockwise oriented closed contour in $\Omega$ that encloses $\sigma(X)$.\\ 
 iii) \cite{2} The matrix $f(X)$ is said to be a \emph{primary matrix function}. 
    \end{defi}
    \begin{rem}
 The matrix $f(X)$ does not depend on the choice of the contour $\Gamma$.  
  \end{rem}
 We have the following well-known result
  \begin{prop} \label{primary} 
 i) The matrix $f(X)$ can be written as a polynomial in $X$ whose coefficients depend on $X$.\\
 ii) We have the equality $\sigma(f(X))=f(\sigma(X))$.
 \end{prop}
 \begin{proof}
 See~\cite{2}.
 \end{proof}

    \begin{thm} \label{ST}
    Let $A,X\in\mathcal{M}_n(\mathbb{C})$ be such that $XA-AX=f(X)$. The matrices $A$ and $X$ are simultaneously triangularizable and $\sigma(X)\subset{f}^{-1}(0)$.
\end{thm}
\begin{proof}
Let $V$ be the vector space spanned by $\{A,I_n,X,\cdots,X^{n-1}\}$. One checks easily by induction that:
\begin{equation}\label{relation rec}  \text{for all}\; i\geq{1},\;\; X^iA-AX^i=iX^{i-1}f(X). \end{equation}
 By Cayley-Hamilton's Theorem and Proposition \ref{primary}, $X^iA-AX^i$ belongs to $V$, and $V$ is a Lie algebra. The derived series of $V$ is 
 $$V_1=[V,V]\subset\mathbb{C}[X], \;V_2=[V_1,V_1]=\{0\}.$$
 Thus $V$ is solvable. According to Lie's Theorem, $V$ is triangularizable, that is $A$ and $X$ are simultaneously triangularizable. Therefore, $XA-AX$ is a nilpotent matrix and $f(\sigma(X))=\{0\}$ (see Proposition \ref{primary}).
\end{proof}
\begin{defi}
 Let $M=[m_{ij}]$ be a strictly upper triangular $(n\times{n})$ matrix. For every $i\in\{1,\cdots,n-1\}$ the set $\{m_{1,i+1},\cdots,m_{n-i,n}\}$ is said to be the \emph{false diagonal of $M$ with index $i$}.
 \end{defi} 
\begin{thm} \label{valpr}
Assume that there exists a unique $\alpha\in{\Omega}$ such that $f(\alpha)=0$. Equation (\ref{equation 1}) admits a solution $X$ such that $X\not=\alpha{I}_n$ if and only if there exist $\lambda,\mu\in\sigma(A)$ such that $\lambda-\mu=f'(\alpha)$.
\end{thm}
\begin{proof}
Let $X\in\mathcal{M}_n(\mathbb{C})$ be a solution of Equation (\ref{equation 1}). According to Theorem \ref{ST}, $\sigma(X)\subset{f}^{-1}(0)=\{\alpha\}$. Thus $X=\alpha{I}_n+N$ where $N$ is a nilpotent matrix. Replacing $X$ by $\alpha{I}_n+N$ in Equation (\ref{equation 1}), one obtains the following equivalences
\begin{eqnarray}
X\text{ is solution of Equation (\ref{equation 1})} &\Leftrightarrow& NA-AN=f(\alpha{I}_n+N) \nonumber\\
&\Leftrightarrow&   NA-AN=f'(\alpha)N+\cdots+\dfrac{f^{(n-1)}(\alpha)}{(n-1)!}N^{n-1}.\label{taylor}
\end{eqnarray}
Now we prove that Equation (\ref{taylor}) admits a non-zero nilpotent solution $N$ if and only if there exist $\lambda,\mu\in\sigma(A)$ such that $\lambda-\mu=f'(\alpha)$.\\
($\Leftarrow$). We may assume that $A=[a_{ij}]$ is a upper triangular matrix such that $a_{11}=\mu$ and $a_{nn}=\lambda$.
We consider the non-zero nilpotent matrix $N=\begin{pmatrix}
&1\\
\boxed{\begin{matrix}&\\&&0&&\\&\end{matrix}}
&
\end{pmatrix}$.
Since $N^2=0$, one has 
$$NA-AN=(\lambda-\mu)N=f(\alpha{I}_n+N).$$
($\Rightarrow$). We may assume that $A$ and $N$ are upper triangular matrices such that $N\not=0$ and, for every $i$, $a_{ii}=\lambda_i$. Suppose that
\begin{equation} \label{ineg} \text{for all}\; i\not=j,\;\; \lambda_i-\lambda_j\not=f'(\alpha).  \end{equation}
 By considering the non-zero false diagonal of $N$ with minimal index $k$, Equation (3) gives the following relations:
 $$\begin{array}{ccl}
 n_{1,1+k}(\lambda_{k+1}-\lambda_1-f'(\alpha))&=&0\\
 \vdots&\vdots&\vdots\\
 n_{n-k,n}(\lambda_{n}-\lambda_{n-k}-f'(\alpha))&=&0.
 \end{array}$$
 According to Inequality (\ref{ineg}), the false diagonal of index $k$ is zero. That is a contradiction.
\end{proof}
%//////////////////////////////////////////////////////////////////////////////////////////////////////////////////////////////////////////////////
\section{The case $f'(\alpha)=0$}
Now, we assume that there exists a unique $\alpha\in{\Omega}$ such that $f(\alpha)=0$ and $f'(\alpha)=0$. Then we look for the non-zero nilpotent solutions of the equation:\\
\begin{equation}  XA-AX=X^pg(X)  \label{xg} \end{equation}
 where $p\in\{2,\cdots,n-1\}$ and $g$ is a polynomial depending on $f$ only, such that $g(0)\not=0$ and $\deg(g)<n-p$.
 Relation (\ref{relation rec}) can be rewritten as 
 \begin{equation} \label{xgrecu} \text{for all}\;\;i\geq{1},\;\; X^iA-AX^i=iX^{i+p-1}g(X).  \end{equation}
 \begin{rem}
 i) According to Theorem \ref{valpr}, necessarily $A$ has multiple eigenvalues.\\ 
 ii) The case $p=n$ reduces to finding the nilpotent matrices that commute with $A$.
  \end{rem}
 \begin{defi}
 Let $h(X)=\sum_{i=0}^ka_iX^i$ be a non-zero polynomial. The valuation of $h$ is
   $val(h)=\min\{i\;|\;a_i\not=0\}$.
   \end{defi}
   \begin{lem} \label{lem asxl}
 Let $s,l\in\mathbb{N}^*$. The following equality holds:
 \begin{equation} \label{asxl}  A^sX^l=\sum_{j=0}^{s}h_j(X)A^{s-j} \end{equation}
 where, for every $0\leq{j}\leq{s},\;h_j$  is a polynomial such that $val(h_j)\geq{l}+j(p-1)$.  
 \end{lem}
 \begin{proof}
  The proof is by induction on $s$. The case $s=1$ is clear because, according to Relation (\ref{xgrecu}), $AX^l=X^lA-lX^{l+p-1}g(X)$. We assume that Relation (\ref{asxl}) is true. Then we have 
 $$A^{s+1}X^l=\sum_{j=0}^{s}Ah_j(X)A^{s-j}=\sum_{j=0}^{s}AX^{l+j(p-1)}\phi(X)A^{s-j},$$
  where $\phi$ is a polynomial. For every $l,j\in\mathbb{N},j\leq{s}$, thanks to Relation (\ref{xgrecu}), the following equality holds:
 $$AX^{l+j(p-1)}\phi(X)A^{s-j}=[X^{l+j(p-1)}A-(l+j(p-1))X^{l+(j+1)(p-1)}g(X)]\phi(X)A^{s-j}.$$
 Thus it remains to consider $X^{l+j(p-1)}A\phi(X)A^{s-j}$, the first part of the RHS of the previous expression  or, by linearity, an expression in the form $X^{l+j(p-1)}AX^qA^{s-j}$
 where $q\in\mathbb{N}$. Again thanks to Relation (\ref{xgrecu}),
 \begin{eqnarray*}
 X^{l+j(p-1)}AX^qA^{s-j}&=&X^{l+j(p-1)}(X^qA-qX^{q+p-1}g(X))A^{s-j}\\
 &=& X^{l+j(p-1)+q}A^{(s+1)-j}-qX^{l+(j+1)(p-1)+q}g(X)A^{(s+1)-(j+1)}.
 \end{eqnarray*}
 One deduces that for every $0\leq{j}\leq{s}$
 $$AX^{l+j(p-1)}\phi(X)A^{s-j}=\sum_{t=0}^{s+1}g_t(X)A^{s+1-t},$$
 where each $g_t$ is a polynomial satisfying $val(g_t)\geq{l}+t(p-1)$.
  \end{proof}
 \begin{lem} \label{amv}
 Assume that the matrices $A$ and $X$ satisfy Equation (\ref{xg}) and that $X$ is a nilpotent matrix. Let $v\in\mathbb{C}^n$ and $m\in\mathbb{N}^*$ such that $A^mv=0$, $r,k$ be integers such that $r\geq{n}$ and $1\leq{k}<\dfrac{r}{p-1}$. For every $t\in\mathbb{N}$ such that $t\geq{r}-k(p-1)$, one has $A^{m+k-1}X^tv=0$. 
 \end{lem}
 \begin{proof}
 The proof is by induction on $k$. Let $t\geq{r-p+1}$. By Relation (\ref{xgrecu}), $AX^t-X^tA=-tX^{t+p-1}g(X)=0$. Thus, $A$ and $X^t$ commute and $A^mX^tv=X^tA^mv=0$. This proves the case $k=1$. \\
 We assume that Lemma \ref{amv} is true for $k-1$. Let $t\geq{r}-k(p-1)$. By Relation (\ref{xgrecu}), 
 \begin{equation} \label{noyau} A^{m+k-1}X^tv=A^{m+k-2}X^tAv-tA^{m+k-2}X^{t+p-1}g(X)v.\end{equation}
    Thanks to the induction hypothesis and to the inequality $t+p-1\geq{r}-(k-1)(p-1)$, one has
    $$A^{m+k-2}X^{t+p-1}g(X)v=0.$$
 By Lemma \ref{lem asxl}, we can write 
 $$A^{m+k-2}X^tAv=\sum_{j=0}^{m+k-2}h_j(X)A^{m+k-j-1}v,$$
  with for every $j\in\llbracket{0},m+k-2\rrbracket$, $val(h_j)\geq{t}+j(p-1)$.\\
  For every $j\geq{k}$ one has 
  $$val(h_j)\geq{t}+j(p-1)\geq{t}+k(p-1)\geq{r}\geq{n}.$$
   Hence, for all $j\geq{k}$, one has $h_j(X)=X^n\phi(X)=0$ where $\phi$ is a polynomial. Obviously, if $j<k$ then one has $m+k-j-1\geq{m}$ and $A^{m+k-j-1}v=0$. According to Relation (\ref{noyau}), we are done.
   \end{proof} 
 \begin{thm} \label{general}
 Assume that the matrix $X$ satisfies Equation (\ref{xg}) and that $X$ is a nilpotent matrix. Then the generalized eigenspaces of $A$ are $X$-invariant.
 \end{thm}
 \begin{proof}
 Let $\lambda$ be an eigenvalue of $A$ and let $H_\lambda=ker(A-\lambda{I}_n)^n$ be the generalized eigenspace associated to $\lambda$. We may assume $\lambda=0$. For every $v\in{H}_0$, there exists an integer $m$ such that $A^mv=0$. Let $k\in\mathbb{N}^*$ such that $r=1+k(p-1)\geq{n}$. Applying Lemma \ref{amv}, with $k=\dfrac{r-1}{p-1}$ and $t=1$, we obtain $A^{m+k-1}Xv=0$, that is $Xv\in{H}_{0}$. 
 \end{proof}
 The generalized eigenspaces of $A$ span $\mathbb{C}^n$ and thus we can deduce easily the following result
 \begin{cor} \label{corgene}
 Let $\sigma(A)=\{\lambda_1,\cdots,\lambda_k\}$ and $P$ be an invertible matrix such that $$P^{-1}AP=\bigoplus_{i=1}^k(\lambda_iI_{\alpha_i}+N_i),$$
  where $\sum_{i=1}^k\alpha_i=n$ and, for all $i$, $N_i$ is an $(\alpha_i\times\alpha_i)$ nilpotent matrix.\\
\indent If $X$ is a nilpotent matrix solution of Equation (\ref{xg}), then $P^{-1}XP=\bigoplus_{i=1}^k\limits{X_i}$, where for every $i$, $X_i$ is a nilpotent matrix that satisfies $X_iN_i-N_iX_i=X_i^pg(X_i)$.  
 \end{cor}
 %//////////////////////////////////////////////////////////////////////////////////////////////////////////////////////////////////////////////////
\section{The case $f'(\alpha)=0$ and $A$ non-derogatory}
\begin{defi}
A complex square matrix $M$ is said to be \emph{non-derogatory} if, in its Jordan normal form, for all $\lambda\in\sigma(M)$, the number of Jordan blocks associated with $\lambda$ is $1$.
\end{defi}
We assume $A$ is non-derogatory and we consider Equation (\ref{xg}). According to Corollary \ref{corgene}, it is enough to solve the equation:
\begin{equation} \label{jordan} XJ_n-J_nX=X^pg(X),  \end{equation}
 where $J_n$ is the nilpotent Jordan block of dimension $n$ and $X$ is a unknown nilpotent matrix.
\subsection{The case where $A$ is a Jordan block}
\begin{lem} \label{lemjordan}
Let $X$ be a nilpotent matrix that is solution of Equation (\ref{jordan}) and let $(e_i)_{1\leq{i}\leq{n}}$ be the canonical basis of $\mathbb{C}^n$. Then $Xe_1=0$.
\end{lem}
\begin{proof}
$\bullet$ We show that for $l\ge{p}$, $X^le_1=0$ implies $X^{l-(p-1)}e_1=0$. Indeed, one has
 $$J_nX^{l-(p-1)}-X^{l-(p-1)}J_n=-(l-(p-1))X^lg(X).$$ 
 Thus we deduce that $J_nX^{l-(p-1)}e_1=0$. Therefore, there exists $\lambda\in\mathbb{C}$ such that $X^{l-(p-1)}e_1=\lambda{e}_1$. Since $X$ is nilpotent,  $\;\lambda=0$.\\
$\bullet$ Obviously, $X^ne_1=0$. By repeating the previous argument, there exists $k\leq{p}$ such that $X^ke_1=0$. Thus $X^pe_1=0$ and $Xe_1=0$.
\end{proof}
\begin{lem} \label{lemtriang}
Every nilpotent solution of Equation (\ref{jordan}) is strictly upper triangular.
\end{lem}
\begin{proof}
The proof is by induction on $n$. The result is obvious for $n=1$. Let $X$ be a nilpotent solution of Equation (\ref{jordan}). According to Lemma \ref{lemjordan}, $X=\begin{pmatrix}0&*\\0&X_1\end{pmatrix}$  where $X_1$ is a $(n-1\times{n}-1)$ matrix. Clearly, $X_1$ is nilpotent and satisfies 
$$X_1J_{n-1}-J_{n-1}X_1=X_1^pg(X_1).$$
 Thus, $X_1$ is strictly upper triangular by induction hypothesis.
\end{proof}
\begin{thm} \label{thmalgo}
We can obtain all nilpotent solutions $X=[x_{ij}]$ of Equation (\ref{jordan}) in the following way. Let $k$ be the coefficient of degree two of the polynomial $X^pg(X)$. We choose arbitrarily the last column of $X$ with $x_{n-1,n}$ satisfying:
\begin{equation}  \label{cond} \text{for all }i\in\{1,\cdots,n-2\},\;\;1-ikx_{n-1,n}\not=0.  \end{equation}
The other coefficients of the strictly upper triangular matrix $X$ are obtained recursively by solving equations of degree one.
\end{thm}
\begin{proof}
$\bullet$ The computation of the false diagonal of index $2$ in Equation (\ref{jordan}) gives:
 $$\text{for all }2\leq{i}\leq{n-1},\;\; x_{i-1,i}(1-kx_{i,i+1})=x_{i,i+1}.$$
 If there exists an $i\in\llbracket{2},n-1\rrbracket$ such that $1-kx_{i,i+1}=0$  then Equation (\ref{jordan}) admits only the zero solution. Otherwise, one has for all $i\in\llbracket{2},n-1\rrbracket$,
 $$x_{i-1,i}=\dfrac{x_{i,i+1}}{1-kx_{i,i+1}}.$$
 Hence if Condition (\ref{cond}) is satisfied, then for all $i\in\llbracket{2},n-1\rrbracket$, one has:
 $$x_{i-1,i}=\dfrac{x_{n-1,n}}{1-(n-i)kx_{n-1,n}}.$$ 
  $\bullet$ We reason in the same way for the false diagonals of index $3,4,\cdots,n-2$. Finally, for the last false diagonal, $x_{1,n}$ can be arbitrarily chosen.
\end{proof}
\begin{rem} 
 i) Once the last column of $X$ is chosen, the matrix $X$ is uniquely determined.\\
 ii) The obtained matrix $X$ is similar to $J_n$ if and only if $x_{n-1,n}\not=0$. 
 \end{rem}
 \subsection{Numerical computations}
 The performance of computations was measured by using  a 2 GHz Intel Core Duo processor provided with 2 GB RAM.
One checks easily that the complexity of the calculations is $O(n^2)$.
On the other hand, the determination of the coefficients of $X$ as a function of the last column, is much more complicated.
For instance, for $n=10$, we consider the equation 
$$XJ_{10}-J_{10}X=-7X^2+3X^3-X^4+X^5-2X^6-X^7+3X^8-5X^9$$
 where $X$ is a nilpotent matrix.\\
$First\; case$. We look for the matrix $X=[x_{ij}]$ as a function of the $(x_{i,10})_{1\leq{i}\leq{9}}$. Using Maple, we obtain for every $1\leq{i}<j\leq{9}$, $x_{ij}=\dfrac{P_{ij}(x_{1,10},\cdots,x_{9,10})}{Q_{ij}(x_{9,10})}$, where $P_{ij},Q_{ij}$ are polynomials. The duration of calculation is $12$ seconds but the display requires more than $2000$ Maple lines.\\
$Second\; case$. For every $1\leq{i}\leq{9}$, $x_{i,10}$ is randomly chosen as integer in $\llbracket{-}10,10\rrbracket$. We obtain the exact values of the entries of $X$, as rational fractions, in $0.6$ second.
\subsection{$A$ is non-derogatory}
Thanks to Corollary \ref{corgene} and Theorem \ref{thmalgo}, we obtain easily the following result concerning a matrix $A$ such that $\sigma(A)=\{\lambda_1,\cdots,\lambda_k\}$.
\begin{thm} \label{derog}
Assume that $A=[a_{ij}]$ is non-derogatory. All the nilpotent solutions of Equation (\ref{xg}) can be explicitly determined. Moreover a nilpotent solution $X$ of Equation (\ref{xg}) has coefficients in  $\mathbb{Q}((a_{ij})_{ij},(\lambda_i)_{1\leq{i}\leq{k}},(u_{1,\alpha_i},\cdots,u_{\alpha_i-1,\alpha_i})_{1\leq{i}\leq{k}})$ where the $(u_{1,\alpha_i},\cdots,u_{\alpha_i-1,\alpha_i})_{1\leq{i}\leq{k}}$ are chosen as last columns during the calculations of  Theorem 4. In particular, the general solution $X$ depends on $n-k$ parameters.
\end{thm}
%//////////////////////////////////////////////////////////////////////////////////////////////////////////////////////////////////////////////////
\section{The case $f'(\alpha)\not=0$}
 Now we assume that there exists a unique $\alpha\in{\Omega}$ such that $f(\alpha)=0$ and that $f'(\alpha)\not=0$. We study the non-zero nilpotent solutions of the equation
\begin{equation} \label{nonzero} XA-AX=Xg(X),  \end{equation}
 where $g$ is a polynomial such that $\deg(g)<n-1$ and $g(0)\not=0$. Moreover, $g$ depends only on the function $f$. We may assume $g(0)=1$. Relation (\ref{relation rec}) can be rewritten in the following manner: 
 \begin{equation}  \label{rec}  \text{for all  }i\geq{1},\;\; X^iA-AX^i=iX^ig(X).  \end{equation}
  \begin{rem}
 For $g=1$ we obtain a particular case of the Sylvester equation (see~\cite{2})
 \begin{equation} \label{sylv}  \phi{(}X)=XB-CX=D   \end{equation}
  where $B\in\mathcal{M}_{q}(\mathbb{C})$, $C\in\mathcal{M}_{p}(\mathbb{C})$ and $D\in\mathcal{M}_{pq}(\mathbb{C})$ are given and $X\in\mathcal{M}_{pq}(\mathbb{C})$ is to be determined. Note that\\
 \begin{equation}  \label{sylves}   \sigma(\phi)=\{\lambda-\mu|\lambda\in\sigma(B),\mu\in\sigma(C)\}.  \end{equation} 
 \end{rem}
 \begin{lem}   \label{invar}
 Let $k\in\mathbb{N}^*$. Then $ker(X^k)$ is $A$-invariant.
 \end{lem}
 \begin{proof}
 Let $u\in{k}er(X^k)$. The equality 
 $$X^kA-X^{k-1}AX=X^kg(X)$$
  implies that $X^kAu=X^{k-1}AXu$.
   With Equality (\ref{rec}), we obtain $X^{k-1}AXu=0$. 
 \end{proof}
 \subsection{Decomposition of the solutions}
 Let $A\in\mathcal{M}_n(\mathbb{C})$ and $X$ be a nilpotent solution of Equation (\ref{nonzero}). Let $\lambda\in\sigma(A),u\in\mathbb{C}^n$ and $p\in\llbracket{1},n\rrbracket$ such that $(A-\lambda{I}_n)^pu=0$. 
 \begin{lem} \label{lemdecomp}
 Let $s$ be a positive integer. One has $$i)\;\;Xu\in\bigoplus_{i=1}^{n-1}ker(A+(i-\lambda)I_n)^p.$$
 $$ii)\;\;X^su\in\bigoplus_{i\geq{s}}ker(A+(i-\lambda)I_n)^p.$$ 
 \end{lem}
 \begin{proof}
 We may assume $\lambda=0$.
 For any integers $r,s$, $P_s^{(r)}$ denotes a polynomial in the variable $X$ such that $val(P_s^{(r)})\geq{s}$. We prove the equalities
\begin{eqnarray}
a)\;\;\;\;\;\;\; \text{ for }k\leq{n},\;\;(A+kI_n)X^k&=&X^kA+P_{k+1}^{(0)}, \label{decomp}\\
b)\;\; \text{ for  }k<l\leq{n},\;\;(A+kI_n)X^l&=&X^lA+P_l^{(0)}, \nonumber\\
c) \text{ for  }k<l\leq{n},\;\;(A+kI_n)P_l^{(0)}&=&P_l^{(1)}A+P_l^{(2)}.\nonumber  
\end{eqnarray}
By Relation (\ref{rec}), $AX^k+kX^k=X^kA+kX^k(1-g(X))$. Since $g(0)=1$, this shows $a)$.
In the same way, the equality $AX^l+kX^l=X^lA+X^l(k-lg(X))$ shows $b)$. Finally, by linearity, we deduce $c)$ from $b)$.\\
$\bullet$ Using Equalities (\ref{decomp}), we easily see that $$(A+I_n)^pXu=(XA^p+\sum_{r=0}^{p-1}P_2^{(r)}A^r)u=\sum_{r=0}^{p-1}P_2^{(r)}A^ru.$$
By induction on $s$, we obtain for every $s\in\mathbb{N}^*$:
\begin{eqnarray*}
 (A+sI_n)^p\cdots(A+I_n)^pXu&=&(X^sA^pQ(A)+\sum_{r=0}^{p-1}P_{s+1}^{(r)}A^r)u\\
 &=&\sum_{r=0}^{p-1}P_{s+1}^{(r)}A^ru, \end{eqnarray*}
  where $Q$ is a polynomial. We choose $s=n-1$. Since for all $r\in\mathbb{N}$, $P_n^{(r)}=0$, this shows $i)$ above.\\
 $\bullet$  The proof of $ii)$ is by induction on $s$. Assume that $X^su=\sum_{i\geq{s}}u_i$, where for every $i$, $u_i\in{k}er(A+iI_n)^p$. Then $X^{s+1}u=\sum_{i\geq{s}}Xu_i$ where, by $i)$, $Xu_i\in\bigoplus_{j>i}\limits{k}er(A+jI_n)^p$.    
 \end{proof}
 \begin{lem} \label{lemnilp}
 Let $s$ be a positive integer such that $\lambda-s\notin\sigma(A)$. Then 
 $$X^su=0\;\; \text{and}\;\; Xu\in\bigoplus_{i=1}^{s-1}ker(A+(i-\lambda)I_n)^p.$$ 
 \end{lem}
 \begin{proof}
 We may assume $\lambda=0$.\\ 
 $\bullet$ Suppose $X^su\not=0$. Let $k\geq{s}$ be the integer such that $u\in{k}er(X^{k+1})\backslash{k}er(X^k)$. 
Then $(A+kI_n)^pX^ku=\sum_{r=0}^{p-1}P_{k+1}^{(r)}A^ru$ and, by Lemma \ref{invar}, 
\begin{equation} 
X^ku\in{k}er(A+kI_n)^p. \label{ker1}
\end{equation}
 On the other hand, by Lemma \ref{lemdecomp}, $X^su\in\bigoplus_{i\geq{s}}\limits{k}er(A+iI_n)^p=\bigoplus_{i>s}\limits{k}er(A+iI_n)^p$. In the same way, \begin{equation}
 X^ku=X^{k-s}X^su\in\bigoplus_{i>k}\limits{k}er(A+iI_n)^p. \label{ker2}
 \end{equation}
By Relations (\ref{ker1}) and (\ref{ker2}), we deduce that $X^ku=0$, which is a contradiction.\\
 $\bullet$ Since $X^su=0$, one has   $(A+(s-1)I_n)^p\cdots(A+I_n)^pXu=\sum_{r=0}^{p-1}P_{s}^{(r)}A^ru=\sum_{r=0}^{p-1}\phi_{s}^{(r)}X^sA^ru=0$ where each $\phi_{s}^{(r)}$ is a polynomial in $X$. 
 \end{proof}
 \textbf{Notation}. 
 Let $A\in\mathcal{M}_n(\mathbb{C})$. We can write $\sigma(A)=\bigsqcup_{r=1}^k\limits{B}_r$ where the sequence $(B_r)_{1\leq{r}\leq{k}}$ satisfies the following:\\
 $i)$ for every $1\leq{r}\leq{k}$, there exists $\lambda_r\in\mathbb{C},c_r\in\mathbb{N}$ such that $B_r=\llbracket{0},c_r\rrbracket+\lambda_r$.\\
 $ii)$ If $r\not=s$ and $u\in{B}_r,v\in{B}_s$, then $u-v\not=1$.\\
  
 We consider the ordering of the eigenvalues of $A$ induced by the sequence\\ $(B_r)_{1\leq{r}\leq{k}}$ and the associated Jordan normal form of $A$: there exists an invertible matrix $P$ such that $P^{-1}AP=\bigoplus_{r=1}^k\limits{U}_r$ where for every $r$, $\sigma(U_r)=B_r$ and $U_r$ is a Jordan matrix.
 \begin{thm} \label{thmdecomp}
 Let $A\in\mathcal{M}_n(\mathbb{C})$ and $X$ be a nilpotent solution of Equation (\ref{nonzero}). With the previous notation, $P^{-1}XP=\bigoplus_{s=1}^k\limits{X}_s$, where for every $s$, $X_s$ is a nilpotent upper triangular matrix that satisfies $X_sU_s-U_sX_s=X_sg(X_s)$. 
 \end{thm}
 \begin{proof}
 This follows from Lemma \ref{lemdecomp} and Lemma \ref{lemnilp}. 
 \end{proof}
   \subsection{The complete solution.}
  \begin{rem}
  According to Theorem \ref{thmdecomp}, we have reduced the problem to solving Equation (\ref{nonzero}) to the case $A=U_s$.
  \end{rem}
 In the following, we show that if $A=U_s$ then it remains to solve a sequence of Sylvester equations (cf. Equation~(\ref{sylv}) and Property~(\ref{sylves})). Moreover if $U_s$ is diagonalizable, then we obtain an explicit solution that is computable by iteration.\\ 
  First we consider the general case and we may assume that $\sigma(A)=\llbracket0,k-1\rrbracket$. 
 \begin{prop}   \label{equagen}
 Let $A=diag(n_1,I_{s_2}+n_2,\cdots,(k-1)I_{s_k}+n_k)$ where each $n_i$ is a nilpotent matrix of dimension $s_i$ and where $\sum_{i=1}^ks_i=n$. Let $\alpha_2,\cdots,\alpha_{n-1}$ be complex numbers. 
The general nilpotent solution of the equation
  \begin{equation}   XA-AX=X+\sum_{i=2}^{n-1}\alpha_i{X}^i    \label{decompdiag}  \end{equation} 
 is a strictly upper triangular $(k\times{k})$ block matrix $X$. Let $1\leq{i}\leq{k-1}$. Knowing the false block diagonals of $X$ with indices $j<i$, the false block diagonal of $X$ with index $i$ can be obtained by the resolution of $k-i$ similar Sylvester equations.
\end{prop}
\begin{proof}
By Lemma \ref{lemnilp}, $X=[x_{ij}]$ is a strictly upper triangular $(k\times{k})$ block matrix and satisfies $X^k=0$. By identification, the false diagonal of $X$ with index $1$ satisfies 
$$\text{for every }j\in\llbracket1,k-1\rrbracket,\; x_{j,j+1}n_{j+1}-n_jx_{j,j+1}=0.$$
Hence, $x_{j,j+1}$ is any element of $ker(\phi)$, where $\phi$ is the nilpotent Sylvester operator $x\rightarrow{x}n_{j+1}-n_jx$.
 Now the false diagonal of $X$ with index $i\in\llbracket 2,k-1\rrbracket$ satisfies 
$$\text{for every }j\in\llbracket1,k-i\rrbracket, \;x_{j,j+i}((i-1)I_{s_{j+i}}+n_{j+i})-n_jx_{j,j+i}=\psi_{ij}$$
 where $\psi_{ij}$ depends on $(\alpha_k)_{2\leq{k}\leq{i}}$, and on the false diagonals of $X$ with indices in $\llbracket 1,i-1\rrbracket$. These Sylvester equations are in the form $\phi(x_{j,j+i})=\psi_{ij}$ with $\phi=(i-1)Id+\nu$ where $\nu$ is a nilpotent operator. Thus $x_{j,j+i}=(\dfrac{1}{i-1}Id-\dfrac{1}{(i-1)^2}\nu+\dfrac{1}{(i-1)^3}\nu^2-\cdots)\psi_{ij}$.  
\end{proof} 
In the case where $A$ is diagonalizable, one has the following result
\begin{prop}   \label{equadiag}
 Let $A=diag(0_{s_1},I_{s_2}\cdots,(k-1)I_{s_k})$ with $s_1+\cdots+s_k=n$ and $\alpha_2,\cdots,\alpha_{n-1}$ be complex numbers. 
The general nilpotent solution of Equation (\ref{decompdiag})
 is a strictly upper triangular $(k\times{k})$ block matrix $X=[x_{ij}]$ such that each $x_{i,i+1}$ is an arbitrary $(s_i\times{s}_{i+1})$ matrix and, for every $r>1$, $$x_{i,i+r}=P_r(\alpha_2,\cdots,\alpha_r)\,x_{i,i+1}\,x_{i+1,i+2}\,\cdots\,{x}_{i+r-1,i+r}$$ 
 where $P_r$ is a polynomial in $\alpha_2,\cdots,\alpha_r$, with coefficients in $\mathbb{Q}$, that depends only on $r$.
 \end{prop}
 \begin{proof}
 By Lemma \ref{lemnilp}, $X$ satisfies $X^k=0$ and is a strictly upper triangular $(k\times{k})$ block matrix $X=[x_{ij}]$. Let $i\in\llbracket1,k-1\rrbracket$. We obtain, by identification,
$$\text{ for every }r\in\llbracket1,k-i\rrbracket, (r-1)x_{i,i+r}=\sum_{s=2}^r\alpha_s\sum_{i<i_1<\cdots<i_{s-1}<i+r}x_{i,i_1}x_{i_1,i_2}\cdots{x}_{i_{s-1},i+r}.$$ 
Thus the matrices $(x_{i,i+1})_{1\leq{i}\leq{k-1}}$ are arbitrary. Moreover $x_{i,i+2}=\\\alpha_2\,(x_{i,i+1}x_{i+1,i+2})$ or $P_2(\alpha_2)=\alpha_2$. Obviously $x_{i,i+r}$ is expressed as a function of $(x_{j,j+u})_{1\leq{j}<k,1\leq{u}<r}$. Then, by induction on the index $r$ of the false diagonal of $X$, it is easy to show the required formula for $x_{i,i+r}$.
\end{proof}
\begin{exam} \label{example}
 For instance, one has $$P_3(\alpha_2,\alpha_3)=\alpha_2^2+\dfrac{\alpha_3}{2},P_4(\alpha_2,\alpha_3,\alpha_4)=\alpha_2^3+\dfrac{4}{3}\alpha_2\alpha_3+\dfrac{\alpha_4}{3},$$ $$P_5(\alpha_2,\cdots,\alpha_5)=\alpha_2^4+\dfrac{29}{12}\alpha_2^2\alpha_3+\dfrac{7}{6}\alpha_2\alpha_4+\dfrac{3}{8}\alpha_3^3+\dfrac{\alpha_5}{5},$$ $$P_6(\alpha_2,\cdots,\alpha_6)=\alpha_2^5+\dfrac{37}{10}\alpha_2^3\alpha_3+\dfrac{13}{5}\alpha_2^2\alpha_4+\dfrac{8}{5}\alpha_3^2\alpha_2+\dfrac{3}{5}\alpha_3\alpha_4+\dfrac{11}{10}\alpha_2\alpha_5+\dfrac{\alpha_6}{5}.$$
Using Maple, one obtains $(P_{i})_{i\leq{1}3}$ in $1$ minute $50$ seconds.
\end{exam}
  \begin{rem} 
 The solution $X=0$ is always a cluster point of the set of the solutions of Equation (\ref{decompdiag}).
\end{rem}
\subsection{Application to the logarithm function.\\}
Let $A$ be a $(n\times{n})$ complex matrix.
\begin{defi}
Let $X$ be a $(n\times{n})$ matrix that has no eigenvalues on $\mathbb{R}^-=\{x\in\mathbb{R}\;|\;x\leq{0}\}$. The $X$-principal logarithm $\log(X)$ is the $(n\times{n})$ matrix $U$ such that $e^U=X$ and the eigenvalues of $U$ lie in the strip $\{z\in\mathbb{C}\;|\;\Im(z)\in(-\pi,\pi)\}$. 
\end{defi}
\begin{rem}
The function $X\rightarrow\log(X)$ is a primary matrix function.
\end{rem}
\begin{prop}    \label{log}
We consider the $(n\times{n})$ matrices $X$ that have no eigenvalues on $\mathbb{R}^-$ and that satisfy:
 \begin{equation} XA-AX=\log(X).    \label{equalog}  \end{equation}
  The general solution of Equation (\ref{equalog}) is $X=I_n+N$ where $N$ is a nilpotent matrix that satisfies $NA-AN=N+\sum_{i=2}^{n-1}\dfrac{(-1)^{i-1}}{i}N^i$. Moreover, if $A$ is given explicitely, then, using the proof of Proposition \ref{equagen}, we can calculate the solutions $N$. 
\end{prop}
\begin{proof}
We use Proposition \ref{equagen}. with $f(X)=\log(X)$, $\Omega=\mathbb{C}\setminus\mathbb{R}^-$, $\alpha=f^{-1}(0)=1$ and $f'(\alpha)=1$. Then $X=I_n+N$ where $N$ is a nilpotent matrix that satisfies $NA-AN=\log(I_n+N)=N-\dfrac{1}{2}N^2+\dfrac{1}{3}N^3-\cdots$.
 \end{proof}
 \begin{rem}
 If $A$ is diagonalizable, then the values of the polynomials cited in Example \ref{example} are: $P_2=\dfrac{-1}{2},P_3=\dfrac{5}{12},P_4=\dfrac{-31}{72},P_5=\dfrac{361}{720},P_6=\dfrac{-4537}{7200}$.
 \end{rem}
%//////////////////////////////////////////////////////////////////////////////////////////////////////////////////////////////////////////////////
\section{The matrix equation $f(XA-AX)=X$\\}
Of course $XA-AX=\log(X)$ implies $e^{XA-AX}=X$. But is the converse true ? We consider the equations in the form: $f(XA-AX)=X$.
\subsection{The general case}
Let $f:\Omega\rightarrow{\mathbb{C}}$ be an analytic function and $\Omega$ be a complex domain containing $0$. We denote $\alpha=f(0)$. 
\begin{prop}   \label{invfonct}
Let $X$ be a $(n\times{n})$ matrix such that $f(XA-AX)=X$. There exists a nilpotent matrix $N$ such that $X=\alpha{I}_n+N$, $NA-AN$ is nilpotent and 
$$\sum_{k=1}^{n-1}\dfrac{f^{(k)}(0)}{k!}(NA-AN)^k=N.$$
Moreover, if $f'(0)\not=0$, then there exists a ball $\mathcal{B}$ containing $\alpha$ and an analytic function $g$ defined on $\mathcal{B}$ such that $g'(\alpha)\not=0$ and
$NA-AN=\sum_{k=1}^{n-1}\dfrac{g^{(k)}(\alpha)}{k!}N^k$. The derivatives $(g^{(k)}(\alpha))_{1\leq{k}\leq{n-1}}$ are explicitely computable quotients of the known derivatives $(f^{(k)}(0))_{1\leq{k}\leq{n-1}}$.
\end{prop}
\begin{proof}
The matrix $f(XA-AX)$ is a polynomial in $XA-AX$. Clearly, $X$ and $XA-AX$ commute. Thus, $XA-AX$ is a nilpotent matrix, $\sigma(X)=\{\alpha\}$ and there exists a nilpotent matrix $N$ such that $X=\alpha{I}_n+N$ and $f(NA-AN)=\alpha{I}_n+N$.\\
If $f'(0)\not=0$, then $f$ admits a local inverse $g$, an analytic function defined on a neighborhood of $\alpha$ and with values in a neighborhood of $0$. Consequently,
$$NA-AN=g(\alpha{I}_n+N).$$
The last assertion is trivial.
\end{proof}
Taking $f(z)=e^z$, we deduce the following result
\begin{cor}   \label{corexpo}
The equations $e^{XA-AX}=X$ and $XA-AX=\log(X)$ have the same solutions.
\end{cor}
\begin{rem}
Since the nilpotent matrix $N$ commutes with $NA-AN$, the matrices $NA$ and $AN$ are nilpotent. This result is due to Kostant \cite{3}.
\end{rem}
\subsection{The case $f'(0)=0$} We study the equation  
$$f(XA-AX)=X$$ when $f'(0)=0$. We see that, in this case, the properties of the solutions of the previous equation are  
 very different from the case $f'(0)\not={0}$ studied in Proposition \ref{invfonct}.
 \begin{prop}    \label{dim3}
 Let $A$ be a $(3\times{3})$ matrix such that $A$ has three pairwise distinct eigenvalues.
 The equation 
  \begin{equation}  (XA-AX)^2=X   \label{equadim3}  \end{equation} 
  admits non-zero solutions in $\mathcal{M}_3(\mathbb{C})$.
  \end{prop}
  \begin{proof}
 Here $f(x)=x^2$ and $\alpha=0$. By Proposition \ref{invfonct}, $X$ is nilpotent. We may assume $A=diag(u,v,w)$. Let $X=[x_{ij}]$ be a solution of Equation (\ref{equadim3}). Obviously, $X^2=(XA-AX)^4=0$. From the relations $X^2=0\text{ and }(XA-AX)^2=X$, in a coordinatewise way, using Maple, one obtains all the non-zero solutions in the form:\\ $X=\begin{pmatrix}\dfrac{1}{(w-u)(u-v)}&q&qr(u-v)(v-w)\\\dfrac{1}{q(u-v)^2(v-w)(w-u)}&\dfrac{1}{(u-v)(v-w)}&r\\\dfrac{1}{qr(u-v)^2(v-w)^2(w-u)^2}&\dfrac{1}{r(v-w)^2(w-u)(u-v)}&\dfrac{1}{(v-w)(w-u)}\end{pmatrix}$\\
 where $q,r$ are non-zero arbitrary complex numbers.  
  \end{proof}
  \begin{rem}
 $i)$ The matrix $0$ is an isolated point of the set of the solutions of Equation (\ref{equadim3}).\\
 $ii)$ Let $X\not=0$. Since $X$ and $A$ have no common eigenvectors, they are not simultaneously triangularizable.\\
 $iii)$ The matrix $XA-AX$ is a square root of $X$. Hence, solving Equation (\ref{equadim3}) is equivalent to solving the equation $Y^2A-AY^2=Y$ with $X=Y^2$.
  \end{rem}
  \subsection{A mixed equation}
  Let $f,g$ be two analytic functions defined on $\Omega$ that vanish on $0$ only. We show that the equation $f(XA-AX)=g(X)$ may have non-nilpotent solutions.
  \begin{prop}
  Let $A$ be a $(2\times{2})$ matrix such that $A$ has two distinct eigenvalues. 
  The equation 
  \begin{equation}(XA-AX)^2=X^2\label{square}\end{equation}
  admits essentially non-nilpotent solutions in $\mathcal{M}_2(\mathbb{C})$.
   \end{prop}
  \begin{proof}
  We may assume $A=diag(u,v)$. It is easy to show that $\text{Tr}(X)=0$. Hence $X$ is in the form $X=\begin{pmatrix}a&b\\c&-a\end{pmatrix}$, and satisfies the unique relation $a^2+bc((u-v)^2+1)=0$.
  The solution $X$ is nilpotent if and only the supplementary condition $bc=0$ is fulfilled.   
  \end{proof}
  \textbf{Acknowledgments.}
  The author wishes to thank the referee for helpful comments and D. Adam for many valuable discussions.

\bibliographystyle{plain}

\end{document}